\documentclass{amsart}
\usepackage[margin=0.80in]{geometry}
\usepackage{nicematrix,tikz}
\usepackage{xcolor}

\usepackage{amsfonts}
 
\usepackage{setspace}
\usepackage{graphicx}

\usepackage{  xcolor}
\usepackage{amsmath, mathrsfs, amsthm, amscd, amsfonts, amssymb,   color}
\usepackage{url}
\usepackage{enumerate}
\makeatletter
\let\reftagform@=\tagform@
\def\tagform@#1{\maketag@@@{(\ignorespaces\textcolor{blue}{#1}\unskip\@@italiccorr)}}
\renewcommand{\eqref}[1]{\textup{\reftagform@{\ref{#1}}}}
\makeatother
\usepackage{hyperref}
\hypersetup{colorlinks=true, linkcolor=red, anchorcolor=green,
	citecolor=cyan, urlcolor=red, filecolor=magenta, pdftoolbar=true}

\usepackage{epsfig}        
\usepackage{epic,eepic}       

\newtheorem{theorem}{Theorem}
\theoremstyle{plain}

\newtheorem{corollary}{Corollary}

\newtheorem{example}{Example}

\newtheorem{lemma}{Lemma}

\newtheorem{remark}{Remark}

\numberwithin{equation}{section}

\DeclareMathOperator{\tri}{tridiag}
\DeclareMathOperator{\spe}{sp}

\def\etal{et al.\,}

\begin{document}
	\title[Bounding the zeros of polynomials]{Bounding the zeros of polynomials using the Frobenius companion matrix partitioned by the Cartesian decomposition}
	\author[M.W. Alomari ]{Mohammad W. Alomari} 
	\address {Department of Mathematics, Faculty of Science and Information	Technology, Irbid National University,  P.O. Box 2600, Irbid, P.C. 21110, Jordan.}
	\email{mwomath@gmail.com}

	\begin{abstract}
		In this work, some new  inequalities for the numerical radius of block $n$-by-$n$ matrices are presented. As an application, bounding of zeros of polynomials using the Frobenius companion matrix partitioned by the Cartesian decomposition method is proved and affirmed by several numerical examples showing that our approach of bounding zeros of polynomials could be very effective in comparison with the most famous and some recent results presented in the field.  
	\end{abstract}
	
	
	\maketitle	
	
	\section{Introduction}\label{intro}

	Let $\mathscr{H}$ be a complex Hilbert space  with an inner
	product $\langle \cdot ,\cdot \rangle$ and $\mathscr{B(H)}$ be the $C^*$-algebra of all bounded linear
	operators from  $\mathscr{H}$ into itself. When $\mathscr{H} =
	\mathbb{C}^n$, we identify $\mathscr{B}\left( \mathscr{H}\right)$
	with the algebra $\mathscr{M}_{n}$ of $n$-by-$n$ complex
	matrices.  
	For a bounded linear operator $T$ on a Hilbert space
	$\mathscr{H}$, the numerical range $W\left(T\right)$ is the image
	of the unit sphere of $\mathscr{H}$ under the quadratic form $x\to
	\left\langle {Tx,x} \right\rangle$ associated with the operator.
	More precisely,
	\begin{align*}
		W\left( T \right) = \left\{ {\left\langle {Tx,x} \right\rangle :x
			\in \mathscr{H},\left\| x \right\| = 1} \right\}.
	\end{align*}
	Also, the numerical radius is defined to be
	\begin{align*}
		w\left( T \right) = \sup \left\{ {\left| \lambda\right|:\lambda
			\in W\left( T \right) } \right\} = \mathop {\sup }\limits_{\left\|
			x \right\| = 1} \left| {\left\langle {Tx,x} \right\rangle }
		\right|.
	\end{align*}
	
	The spectral radius of an operator $T$ is defined to be
	\begin{align*}
		r\left( T \right) = \sup \left\{ {\left| \lambda\right|:\lambda
			\in \spe\left( T \right) } \right\}.
	\end{align*}
	
	We recall that,  the usual operator norm of an operator $T$ is defined to be
	\begin{align*}
		\left\| T \right\| = \sup \left\{ {\left\| {Tx} \right\|:x \in H,\left\| x \right\| = 1} \right\},
	\end{align*}
	Several numerical radius type inequalities improving and refining
	the inequality   
	\begin{align*}
		\frac{1}{2}\left\|S\right\|\le w\left(S\right) \le \left\|S\right\| \qquad\qquad (S\in \mathscr{B}\left( \mathscr{H}\right))
	\end{align*}
	have been recently obtained by many
	other authors see for example \cite{AF1}--\cite{BF},  and \cite{HD}. 
	Four important
	facts concerning the  numerical radius inequalities   of $n \times
	n$  operator matrices are obtained by different authors which are
	grouped together, as follows:\\
	
	Let  ${\bf S}=\left[S_{ij}\right]\in \mathscr{B}\left(\bigoplus _{i = 1}^n \mathscr{H}_i\right)$ such that $S_{ij}\in\mathscr{B}\left(\mathscr{H}_j, \mathscr{H}_i\right)$. Then
	\begin{align*}
		w\left({\bf S}\right)\le
		\left\{ \begin{array}{l}
			w \left( {\left[ {t_{kj}^{\left( 1 \right)} } \right]} \right),\qquad {\rm{Hou \,\&\, Du \,\,in}\,\,}\text{\cite{HD}}  \\
			\\
			w \left( {\left[ {t_{kj}^{\left( 2 \right)} } \right]} \right) ,\qquad {\rm{BaniDomi \,\&\, Kittaneh \,\,in}\,\,} \text{\cite{BF}}\\ 
			\\
			w \left( {\left[ {t_{kj}^{\left( 3 \right)} } \right]} \right),\qquad  {\rm{AbuOmar \,\&\, Kittaneh \,\,in}\,\,} \text{\cite{AF1}}\\
			\\
			w \left( {\left[ {t_{kj}^{\left( 4 \right)} } \right]} \right) ,\qquad {\rm{AbuOmar \,\&\, Kittaneh \,\,in}\,\,} \text{\cite{AF1}}
		\end{array} \right.;
	\end{align*}
	where
	\begin{align*}
		t_{kj}^{\left( 1 \right)}  &= w \left( {\left[ {\left\| {S_{kj} } \right\|} \right]} \right);
		\qquad
		t_{kj}^{\left( 2 \right)}   = \left\{ \begin{array}{l}
			\frac{1}{2}\left( {\left\| {S_{kj} } \right\| + \left\| {S_{kj}^2 } \right\|^{1/2} } \right),\,\,\,\,\,\,\,k = j
			\\
			\left\| {S_{kj} } \right\|,\qquad\qquad\qquad\,\,\,\,\,\,\,\,\,\,k \ne j
		\end{array} \right.;
		\\
		t_{kj}^{\left( 3 \right)}  &= \left\{ \begin{array}{l}
			w \left( {S_{kj} } \right),\,\,\,\,\,\,\,k = j
			\\
			\left\| {S_{kj} } \right\|,\,\,\,\,\,\,\,\,\,\, k\ne j
		\end{array} \right.;
		\qquad
		t^{\left( 4 \right)}_{kj} =  \left\{ \begin{array}{l}
			w\left( {T_{kj} } \right), \,\,\,\,\,\,\,\,k = j \\
			\\ 
			w\left( {\begin{array}{*{20}c}
					0 & {T_{kj}  }  \\
					{  T_{jk}  } & 0  \\
			\end{array}} \right),\,\,\,\,\,\,\,\,\,k \ne j \\ 
		\end{array} \right..
	\end{align*}
	Clearly, the third and fourth bounds above are gentle refinements of the first and  second bounds, and therefore both $w \left( {\left[ {t_{kj}^{\left( 3 \right)} } \right]} \right)$ and $w \left( {\left[ {t_{kj}^{\left( 4 \right)} } \right]} \right)$
	gives better upper estimates for the numerical radius of ${\bf S}=\left[S_{ij}\right]$.\\
	
	Let $T_n$ be the tridiagonal Toeplitz matrix denoted by $T_n=\tri\left(b,a,c\right)$; i.e.,  
	\begin{align*}
		T_n:=
		\left[\begin{array}{*{20}c}
			a & c & 0  &  \cdots  & 0  
			\\
			b &a &c & \ddots &  \vdots   
			\\
			0 & b&  \ddots  &  \ddots  & 0  
			\\
			\vdots  &  \ddots  &  \ddots  &  \ddots  &c  
			\\
			0  &  \cdots  & 0 & b & a \\
		\end{array}\right]_{n\times n}, \qquad \qquad n\ge2.
	\end{align*}
	It is well known that the eigenvalues of $T_n$ are given by \cite{NPR}:
	\begin{align*}
		\lambda_k = a+2\sqrt{|bc|} \cos\left(\frac{k\pi}{n+1}\right), \qquad k=1,2,\cdots, n
	\end{align*}
	and it have the polar form 
	\begin{align*}
		\lambda_k = a+2\sqrt{|bc|  {\rm{e}}^{i\left( {\theta {\rm{ + }}\phi } \right)/2}\cos\left(\frac{k\pi}{n+1}\right)}, \qquad k=1,2,\cdots, n
	\end{align*}
	where $\theta = \arg\left(b\right)$ and $\phi=\arg\left(c\right)$. In case that, $bc\ne0$, then $T_n$ has $n$ simple eigenvalues, all of them lie in the closed segment
	\begin{align*}
		S_{n,\lambda}=\left\{ a + t 
		{{\rm{e}}^{i\left( {\theta {\rm{ + }}\phi } \right)/2} } :t\in \mathbb{R}, \left|t\right|\le  2\sqrt{|bc|} \cos\left(\frac{\pi}{n+1}\right) \right\} \subset{\mathbb{C}}
	\end{align*}
	The eigenvalues are located symmetrically with respect to $a$. Thus,
	the spectral radius of $T_n$ is given by 
	\begin{align}
		\label{eq1.1}	r\left( {T_n } \right) = \max \left\{ {\left| {a + 2\sqrt {\left| {bc} \right|} 
				{{\rm{e}}^{i\left( {\theta {\rm{ + }}\phi } \right)/2} }
				\cos \left( {\frac{\pi }{{n + 1}}} \right)} \right|,\left| {a + 2\sqrt {\left| {bc} \right|} {\rm{e}}^{i\left( {\theta {\rm{ + }}\phi } \right)/2} \cos \left( {\frac{{n\pi }}{{n + 1}}} \right)} \right|} \right\}. 
	\end{align}
	Moreover, if $bc\ne0$, then the eigenvectors  $x_k=\left[x_{1,n},x _{2,n},\cdots,x_{k,n}\right]^T$ associated
	with the eigenvalue $\lambda_k$ of $T_n$ are given in the form  $x_{k,j}=\left(\frac{c}{b}\right)^{k/2}\sin\left(\frac{kj\pi}{n+1}\right)$, $k,j=1,2,\cdots,n$.
	For comprehensive study about Toeplitz matrices the reader may refer to the interesting book \cite{BG}.\\
	
	The following result is of great interest in the next presented results \cite{NPR}.
	\begin{lemma}
		\label{lem1}	The tridiagonal Toeplitz matrix $T_n=\tri\left(b,a,c\right)$ is normal (i.e., $T^*_nT_n=T_nT^*_n$) if and only if $\left|b\right|=\left|c\right|$. 
	\end{lemma}

	\begin{lemma}
		\label{lem2}	Let $\mathscr{H}_1$ and $\mathscr{H}_2$ be Hilbert spaces, and ${\bf{T}} = \left[ {\begin{array}{*{20}c}
				A & B  \\
				C & D  \\
		\end{array}} \right]$ be an operator matrix with $A\in \mathscr{B}\left(\mathscr{H}_1\right)$, $B\in \mathscr{B}\left(\mathscr{H}_2,\mathscr{H}_1\right)$, $C\in \mathscr{B}\left(\mathscr{H}_1,\mathscr{H}_2\right)$ and $D\in \mathscr{B}\left(\mathscr{H}_2\right)$. Then
		\begin{align*}
			\omega \left( {\bf{T}} \right) \le \frac{1}{2}\left( {\omega \left( A \right) + \omega \left( D \right) + \sqrt {\left( {\omega \left( A \right) - \omega \left( D \right)} \right)^2  + \left( {\left\| B \right\| + \left\| C \right\|} \right)^2 } } \right).
		\end{align*}
	\end{lemma}
	
	\begin{lemma}
		\label{lem3}If $S:=\left[s_{kj}\right]\in \mathscr{M}_n\left(\mathbb{C}\right)$, then
		\begin{align*}
			\omega\left(S\right) \le \omega\left(\left[\left|s_{kj}\right|\right]\right)=\frac{1}{2} r\left(\left[\left|s_{kj}\right|+\left[\left|s_{kj}\right|\right]\right]\right)
		\end{align*} 
	\end{lemma}

	Let  $A\in \mathscr{B}\left( \mathscr{H}\right)  $, then
	\begin{align*}
		\left| {\left\langle {Ax,y} \right\rangle} \right|  ^2  \le \left\langle {\left| A \right|^{2\alpha } x,x} \right\rangle \left\langle {\left| {A^* } \right|^{2\left( {1 - \alpha } \right)} y,y} \right\rangle, \qquad 0\le \alpha \le 1
	\end{align*}
	
	for any   vectors $x,y\in \mathscr{H}$, where  $\left|A\right|=\left(A^*A\right)^{1/2}$. This inequality is well-known as the mixed Schwarz inequality which was introduced in \cite{TK} and generalized later in \cite{FK2}.\\
	
	The following result presents the Cartesian decomposition of the mixed Schwarz inequality \cite{mwo2}.
	\begin{lemma}
		\label{lem4}Let  $A\in \mathscr{B}\left( \mathscr{H}\right)$
		with the Cartesian decomposition    $A=P+iQ$.   If $f$ and $g$ are nonnegative continuous
		functions on $\left[0,\infty\right)$ satisfying $f(t)g(t) =t$
		$(t\ge0)$, then
		\begin{align*}
			\left| {\left\langle {A x,y} \right\rangle } \right| \le   \left\| {f\left( \left|P\right| \right)x}
			\right\|\left\| {g\left( \left|P \right| \right)y} \right\| +
			\left\| {f\left( \left|Q\right| \right)x} \right\|\left\| {g\left(
				\left|Q\right| \right)y} \right\|   
		\end{align*}
		for all $x,y\in \mathscr{H}$.
	\end{lemma}
	
	\section{Numerical radius inequalities of $m\times m$ matrix operator }\label{sec4}
	
	\begin{theorem}
		\label{thm1}    
		Let   ${\bf{A}}=\left[A_{kj}\right]\in \oplus_m \mathscr{M}_{n} \left(\mathbb{C}\right)$ be an $m\times m$ operator  such that $P_{kj}+iQ_{kj}$ is the corresponding  Cartesian decomposition of $A_{kj}$. If $f$ and $g$ are nonnegative continuous functions on $\left[0,\infty\right)$ satisfying $f(t)g(t) =t$ $(t\ge0)$, then
\begin{align}
	w\left({\bf{A}}\right)\le   w^{1/2}\left( \left[c_{kj}\right] \right),\label{eq2.1}
\end{align}
where
\begin{align*}
	c_{kj}= m\cdot \left\{ \begin{array}{l}
		w\left( {P_{kk}^2  + Q_{kk}^2 } \right),\qquad  \qquad\qquad\qquad\qquad\qquad\qquad\qquad\qquad j = k  \\
		\\
		\frac{1}{4}      \left\| {f^2 \left( {\left| {P_{kj} } \right|} \right) + g^2 \left( {\left| {P_{kj} } \right|} \right) + f^2 \left( {\left| {Q_{kj} } \right|} \right) + g^2 \left( {\left| {Q_{kj} } \right|} \right)} \right\|^2, \qquad j \ne k.   \\
	\end{array} \right. 
\end{align*}
\end{theorem}

\begin{proof}
Let $x = \left[ {\begin{array}{*{20}c}  {x_1 } & {x_2 } &  \cdots  & {x_m}  \\
\end{array}} \right]^T \in \bigoplus _{\ell = 1}^m \mathscr{M}_{n} \left(\mathbb{C}\right)$ with $\|x\|=1$,   then we have
\begin{align*}
	&\frac{1}{m^2}\left| {\left\langle {{\bf{A}}x,x} \right\rangle } \right|^2  
	\\
	&= 	\frac{1}{m^2}\left|
	{\sum\limits_{k,j = 1}^m {\left\langle { A_{kj} x_j ,x_k }
			\right\rangle } } \right|^2
	\\
	&\le  \frac{1}{m} \sum\limits_{k,j = 1}^m {\left| {\left\langle { A_{kj}  x_j ,x_k } \right\rangle } \right|^2}  \nonumber\qquad \qquad{\rm{ (by \,\,Jensen's\,\, inequality)}}
	\\
	&=  \frac{1}{m} \sum\limits_{k  = 1}^m {\left| {\left\langle {A_{kk}  x_k , x_k } \right\rangle } \right|^2}  + \frac{1}{m}   \sum\limits_{\scriptstyle j = 1 \hfill \atop
		\scriptstyle j \ne k \hfill}^m {\left| {\left\langle {A_{kj} x_j ,x_k } \right\rangle } \right|^2} \nonumber\\
	&\le  \frac{1}{m}   
	\sum\limits_{k = 1}^m {\left( {\left\langle {P_{kk} x_k ,x_k } \right\rangle ^2  + \left\langle {Q_{kk} x_k ,x_k } \right\rangle ^2 } \right)  } 
	\\
	&\qquad+ \frac{1}{m} \sum\limits_{\scriptstyle j = 1 \hfill \atop
		\scriptstyle j \ne k \hfill}^m { \left[  \left\langle {f^2 \left( {\left| {P_{kj} } \right|} \right)x_j ,x_k } \right\rangle ^{\frac{1}{2}} \left\langle {g^2 \left( {\left| {P_{kj} } \right| } \right)x_j ,x_k }\right\rangle ^{\frac{1}{2}}\right. }
	\\
	&\qquad\qquad\qquad+        \left. {\left\langle {f^2 \left( {\left| {Q_{kj} } \right|} \right)x_j ,x_k} \right\rangle ^{\frac{1}{2}} \left\langle {g^2 \left( {\left| {Q_{kj} } \right|} \right)x_j ,x_k } \right\rangle ^{\frac{1}{2}} } \right]^2
	\\
	&\le \frac{1}{m}  
	\sum\limits_{k = 1}^m {\left( {\left\| {P_{kk} x_k } \right\|^2  + \left\| {Q_{kk} x_k } \right\|^2 } \right)  } 
	\\
	&\qquad +\frac{1}{4m} \sum\limits_{\scriptstyle j = 1 \hfill \atop
		\scriptstyle j \ne k \hfill}^m    \left\{     \left\langle { \left[f^2 \left( {\left| {P_{kj} } \right|} \right)+g^2 \left( {\left| {P_{kj} } \right|} \right)\right]  x_j ,x_k } \right\rangle  +   \left\langle {\left[f^2 \left( {\left| {Q_{kj} } \right|} \right)+g^2 \left( {\left| {Q_{kj} } \right|} \right) \right]x_j ,x_k } \right\rangle     \right\}^2
\end{align*}
\begin{align*}
	&=  \frac{1}{m} \sum\limits_{k = 1}^m {\left( {\left\langle {P_{kk}^2 x_k ,x_k } \right\rangle  + \left\langle {Q_{kk}^2 x_k ,x_k } \right\rangle } \right)  } 
	\\
	&\qquad+\frac{1}{4m}  \sum\limits_{\scriptstyle j = 1 \hfill \atop
		\scriptstyle j \ne k \hfill}^m   \left\| {f^2 \left( {\left| {P_{kj} } \right|} \right) + g^2 \left( {\left| {P_{kj} } \right|} \right) + f^2 \left( {\left| {Q_{kj} } \right|} \right) + g^2 \left( {\left| {Q_{kj} } \right|} \right)} \right\|^2\|x_j\| ^2\|x_k\|^2 
	\\
	&\le  \frac{1}{m} 	\sum\limits_{k = 1}^m {\left\langle {\left( {P_{kk}^2  + Q_{kk}^2 } \right)x_k ,x_k } \right\rangle   }  
	\\
	&\qquad+\frac{1}{4m}  \sum\limits_{\scriptstyle j = 1 \hfill \atop
		\scriptstyle j \ne k \hfill}^m   \left\| {f^2 \left( {\left| {P_{kj} } \right|} \right) + g^2 \left( {\left| {P_{kj} } \right|} \right) + f^2 \left( {\left| {Q_{kj} } \right|} \right) + g^2 \left( {\left| {Q_{kj} } \right|} \right)} \right\|^2\|x_j\|  \|x_k\|
	\\
	&\qquad\qquad\qquad\qquad \qquad\qquad\qquad\qquad \qquad\qquad \qquad\qquad({\text{since  \,\,$\|x_j\|^2\le \|x_j\|,\,\, \forall j$}}) 
	\\
	&=  \frac{1}{m} 
	\sum\limits_{k = 1}^m {\left\| {P_{kk}^2  + Q_{kk}^2 } \right\|  \left\| {x_k } \right\|^2} 
	\\
	&\qquad+\frac{1}{4m}  \sum\limits_{\scriptstyle j = 1 \hfill \atop
		\scriptstyle j \ne k \hfill}^m   \left\| {f^2 \left( {\left| {P_{kj} } \right|} \right) + g^2 \left( {\left| {P_{kj} } \right|} \right) + f^2 \left( {\left| {Q_{kj} } \right|} \right) + g^2 \left( {\left| {Q_{kj} } \right|} \right)} \right\|\|x_j\| \|x_k\| 
	\\
	&= \left\langle {\left[ {c_{kj} } \right]y,y} \right\rangle
\end{align*}
where $y=\left( {\begin{array}{*{20}c}{\left\| {x_1 } \right\|} &
		{\left\| {x_2 } \right\|} &  \cdots  & {\left\| {x_m } \right\|}
		\\  \end{array}} \right)^T$. Taking the supremum over $x  \in
\bigoplus_m \mathscr{M}_{n}$, we obtain the desired result.
\end{proof}

	Particularly, we are interested in the following $2\times 2$ cases:

	\begin{corollary}
		\label{cor1}    If $\bf{A}=\left[ {\begin{array}{*{20}c}
				{ A_{11}  } & { A_{12}  }  \\
				{ A_{21}   } & { A_{22}  }  \\
		\end{array}} \right]$ in $\mathscr{M}_n\bigoplus \mathscr{M}_n$, then
		\begin{align}
			\label{eq2.2}\omega\left( \left[ {\begin{array}{*{20}c}
					{ A_{11}  } & { A_{12} }  \\
					{ A_{21}  } & { A_{22}  }  \\
			\end{array}} \right]\right)
			\le  \sqrt{   \omega \left( {P_{11}^2  + Q_{11}^2 } \right) + \omega \left( {P_{22}^2  + Q_{22}^2 } \right)   +  \sqrt {\left( \omega \left( {P_{11}^2  + Q_{11}^2 } \right) - \omega \left( {P_{22}^2  + Q_{22}^2 } \right) \right)^2+N^2  }  }
		\end{align}
		where  $N=  \left\|    \left| P_{12} \right|   +   \left| Q_{12} \right|      \right\|  + \left\|    \left| P_{21} \right|   +   \left| Q_{21} \right|      \right\|$.	
	\end{corollary}
	\begin{proof}
		From Theorem \ref{thm1}, we have
		\begin{align*}
			&\omega\left( \left[ {\begin{array}{*{20}c}
					{ A_{11} } & {  A_{12}  }  \\
					{ A_{21}   } & { A_{22}  }  \\
			\end{array}} \right]\right) \\&\le \sqrt{2} \cdot 
			\omega^{\frac{1}{2}} \left( \left[ {\begin{array}{*{20}c}
					\begin{array}{l}
						\omega \left( {P_{11}^2  + Q_{11}^2 } \right)
						\\
						\\
					\end{array} & \begin{array}{l}
						\left\| {\left| {P_{12} } \right| + \left| {Q_{12} } \right|} \right\|\\
						\\
					\end{array}  \\
					\left\| {\left| {P_{21} } \right| + \left| {Q_{21} } \right|} \right\|  &       \omega \left( {P_{22}^2  + Q_{22}^2 } \right)  \\
			\end{array}}\right] \right)
			\\
			&=    \sqrt{2} \cdot  r^{\frac{1}{2}}     \left( \left[ {\begin{array}{*{20}c}
					\begin{array}{l}
						\omega \left( {P_{11}^2  + Q_{11}^2 } \right)        \\
						\\
					\end{array} & \begin{array}{l}
						\frac{\left\| {\left| {P_{12} } \right| + \left| {Q_{12} } \right|} \right\|+\left\| {\left| {P_{21} } \right| + \left| {Q_{21} } \right|} \right\|}{2} \\
						\\
					\end{array}  \\
					\frac{\left\| {\left| {P_{12} } \right| + \left| {Q_{12} } \right|} \right\|+\left\| {\left| {P_{21} } \right| + \left| {Q_{21} } \right|} \right\|}{2} &       \omega \left( {P_{22}^2  + Q_{22}^2 } \right)  \\
			\end{array}}\right] \right)
			\\
			&=  \sqrt{   \omega \left( {P_{11}^2  + Q_{11}^2 } \right) + \omega \left( {P_{22}^2  + Q_{22}^2 } \right)   +  \sqrt {\left( \omega \left( {P_{11}^2  + Q_{11}^2 } \right) - \omega \left( {P_{22}^2  + Q_{22}^2 } \right) \right)^2+N^2  }  }
		\end{align*}
		which proves the result.
	\end{proof}
	
	\begin{corollary}
		\label{cor2}    If $\bf{A}=\left[ {\begin{array}{*{20}c}
				{ A_{11}  } & { 0  }  \\
				{ 0   } & { A_{22}  }  \\
		\end{array}} \right]$ in $\mathscr{M}_n\bigoplus \mathscr{M}_n$, then
		\begin{align*}
			\omega\left( \left[ {\begin{array}{*{20}c}
					{ A_{11}  } & {0}  \\
					{ 0  } & { A_{22}  }  \\
			\end{array}} \right]\right)
			\le    \max\left(  \omega \left( {P_{11}^2  + Q_{11}^2 } \right),\omega \left( {P_{22}^2  + Q_{22}^2 } \right) \right)
		\end{align*}
		where $P_{kk} + i Q_{kk}$ is the Cartesian decomposition of $A_{kk}$.
	\end{corollary}
	\begin{proof}
		Setting $A_{12}=0=A_{21}$ in Corollary \ref{cor1} and use the fact that $\omega\left( \left[ {\begin{array}{*{20}c}
				{ A   } & {0}  \\
				{ 0  } & {D  }  \\
		\end{array}} \right]\right)
		\le    \max\left(  \omega \left( {A } \right),\omega \left( {D} \right) \right)$, \cite{HKS}.
	\end{proof}

	\section{Applications for bounding zeros of polynomials}\label{sec5}

	One of the most interesting and useful application of the numerical radius inequalities is to bound zeros of complex polynomials using a suitable partition of the well-known Frobenius companion matrix.	Let 	
	\begin{align}
		\label{eq3.1}   p\left(z\right)=z^n +a_nz^{n-1}+\cdots+a_2z+a_1,  
	\end{align}
	be any polynomial with $a_1\ne0$. The general corresponding  companion matrix is defined as:
	\begin{align}
		\label{eq3.2}C\left( p \right): = \left[ {\begin{array}{*{20}c}
				{ - a_n } & { - a_{n - 1} } &  \cdots  & { - a_2 } & { - a_1 }  \\
				1 & 0 &  \cdots  & 0 & 0  \\
				0 & 1 &  \cdots  & 0 & 0  \\
				\vdots  &  \vdots  &  \ddots  &  \vdots  &  \vdots   \\
				0 & 0 &  \cdots  & 1 & 0  \\
		\end{array}} \right].
	\end{align}
	It is well known that the eigenvalues of $C\left( p \right)$ are exactly the zeros of $p\left(z\right)$, see \cite[p. 316]{HJ}.
	
	Based on some numerical radius estimations of $C\left( p \right)$, several authors paid a serious attention to find various upper bounds of the zeros of $p\left(z\right)$, some famous upper bounds are listed as follow:
	If $\lambda$ is a zero of $p$, then
	\begin{enumerate}

		\item Cauchy \cite{HJ},  obtained the following upper bound
		\begin{align}
			\label{eq3.3}  \left|\lambda\right|\le 1+\max\left\{  \left|a_k\right|: k=1,2,,\cdots, n\right\}.
		\end{align}
		
		\item Carmichael and Mason \cite{HJ}, provided the following estimate  
		\begin{align}
			\label{eq3.4}    \left|\lambda\right|\le \sqrt{1+\sum_{k=1}^n{\left|a_k\right|^2}}. 
		\end{align}

		\item Montel \cite{HJ}, proved the following estimate
		\begin{align}
			\label{eq3.5}   \left|\lambda\right|\le  \max\left\{ 1, \sum_{k=1}^n{\left|a_k\right|}\right\}.
		\end{align}
		
		\item Fujii and Kubo \cite{Fujii} have shown that
		\begin{align}
			\label{eq3.6}    \left|\lambda\right|\le \cos\left(\frac{\pi}{n+1}\right) +\frac{1}{2} \left(\left|a_n\right|+\sum_{k=1}^n{\left|a_k\right|^2}\right).
		\end{align}
		
		\item Abdurakhmanov \cite{Abdurakhmanov} , provided the following estimate
		\begin{align}
			\label{eq3.7}  \left| \lambda  \right| \le \frac{1}{2}\left( {\left| {a_n } \right| + \cos \left( {\frac{\pi }{n}} \right) + \sqrt {\left( {\left| {a_n } \right| - \cos \left( {\frac{\pi }{n}} \right)} \right)^2  + \left( {1 + \sum\limits_{k = 1}^{n - 1} {\left| {a_k } \right|^2 } } \right)^2 } } \right).
		\end{align}
		It seems that Paul and Bag \cite{PB}, didn't notice Abdurakhmanov result; where they provided the same estimate.
		
		\item Linden \cite{L}, provided the following estimate
		\begin{align}
			\label{eq3.8} \left| \lambda  \right| \le \frac{{\left| {a_n } \right|}}{n} + \sqrt {\frac{{n - 1}}{n}\left( {n - 1 + \sum\limits_{k = 1}^n {\left| {a_k } \right|^2 }  - \frac{{\left| {a_n } \right|^2 }}{n}} \right)}. 
		\end{align}

		\item Kittaneh \cite{FK2}, improved  Abdurakhmanov estimate by proving that
		\begin{align}
			\label{eq3.9}         \left|\lambda\right|\le
			\frac{1}{2}\left( {\left| {a_n } \right| + \cos \left( {\frac{\pi }{n}} \right) + \sqrt {\left( {\left| {a_n } \right| - \cos \left( {\frac{\pi }{n}} \right)} \right)^2  + \left( {\left| {a_{n - 1} } \right| - 1} \right)^2  + \sum\limits_{j = 1}^{n - 2} {\left| {a_j } \right|^2 } } } \right).
		\end{align}
		
		\item Abu-Omar and Kittaneh \cite{AF2},  introduced the following estimate
		\begin{align}
			\label{eq3.10} \left| \lambda  \right| \le 
			\frac{1}{2}\left( {\frac{{\left| {a_n } \right| + \alpha }}{2} + \cos \left( {\frac{\pi }{{n + 1}}} \right) + \sqrt {\left( {\frac{{\left| {a_n } \right| + \alpha }}{2} - \cos \left( {\frac{\pi }{{n + 1}}} \right)} \right)^2  + 4\beta } } \right).
		\end{align}
		where $\alpha=\sqrt {\sum\limits_{k = 1}^n {\left| {a_k } \right|^2 } }$ and $\beta=\sqrt {\sum\limits_{k = 1}^{n - 1} {\left| {a_k } \right|^2 } }$.

		\item Al-Dolat \etal, provided the estimate
		\begin{align}
			\label{eq3.11}\left| \lambda  \right| \le \frac{1}{2}\left( {\left| {a_n } \right| + 2\cos \left( {\frac{\pi }{n}} \right) + \sqrt {t^2 \left| {a_n } \right|^2  + \sum\limits_{k = 1}^{n - 1} {\left| {a_k } \right|^2 } }  + \sqrt {1 + \left( {1 - t} \right)^2 \left| {a_n } \right|^2 } } \right) 
		\end{align}
		for $t\in \left[0,1\right]$. In fact, the upper bound above should be rewritten under  taking  `$\min$' over $t\in \left[0,1\right]$, which gives the best value for this estimate.
	\end{enumerate}

	To best of our knowledge, there is no single known method have been used in literature  bounding the zeros of polynomial $p\left(z\right)$ using the Frobenius companion matrix partitioned by the Cartesian decomposition method.

	To apply the numerical radius inequalities established in the previous section to $C\left( p \right)$, we note that, we have a little partition challenge in applying our obtained results because the main diagonal in the presented results require  to have a square sub-matrices. So that,	the usual well-known methods of partitioning the companion matrix  $C\left( p \right)$ in our presented results are useless, this is illustrated clearly in the presented inequalities in the previous section, See for example in \eqref{eq2.2}.

	Our proposed approach is to consider the degree $p\left(z\right)$ in \eqref{eq3.1} to be even with fixed integer $n$ such that $a_1\ne0$. To this end consider the even polynomial 
	\begin{align}
		\label{eq3.12} q\left(z\right)=z^{2n} +a_{2n}z^{2n-1}+\cdots+a_2z+a_1, \qquad  n\ge2,\,\, a_1 \ne 0.
	\end{align}
	Let
	\begin{align}	
		\label{eq3.13} C\left(q\right)=
		\left[
		\begin{array}{c|c}
			\overbrace { \left[\begin{array}{*{20}c}
					{ - a_{2n} } & { - a_{2n - 1} } &  \cdots  & { - a_{n+2} } & { - a_{n+1} }  \\
					1 & 0 &  \cdots  & 0 & 0  \\
					0 & 1 &  \cdots  & 0 & 0  \\
					\vdots  &  \vdots  &  \ddots  &  \vdots  &  \vdots   \\
					0  & 0 &  \ddots  &  1 &  0 
				\end{array} \right]_{n\times n}}^{A_{11}} 
			& \overbrace{\left[\begin{array}{*{20}c}
					{ - a_{n} } & { - a_{n - 1} } &  \cdots  & { - a_{2} } & { - a_{1} }  \\
					0 & 0 &  \cdots  & 0 & 0  \\
					0 & 0 &  \cdots  & 0 & 0  \\
					\vdots  &  \vdots  &  \ddots  &  \vdots  &  \vdots   \\
					0  & 0 &  \ddots  &  0 &  0 
				\end{array}\right]_{n\times n}}^{A_{12}}  \\
			\hline
			\underbrace {\left[\begin{array}{*{20}c}
					0 & 0 &  \cdots  & 0 & 1  \\   
					0 & 0 &  \cdots  & 0 & 0  \\
					0 & 0 &  \cdots  & 0 & 0  \\
					\vdots  &  \vdots  &  \ddots  &  \vdots  &  \vdots   \\
					0  & 0 &  \ddots  & 0 &  0  \\
					0 & 0 &  \cdots  & 0 & 0  \\
				\end{array}\right]_{{n\times n} }}_{A_{21}}
			& \underbrace {\left[\begin{array}{*{20}c}
					0 & 0 &  \cdots  & 0 & 0  \\   
					1 & 0 &  \cdots  & 0 & 0  \\
					0 & 1 &  \cdots  & 0 & 0  \\
					\vdots  &  \vdots  &  \ddots  &  \vdots  &  \vdots   \\
					0  & 0 &  \ddots  & 1 &  0  \\
					0 & 0 &  \cdots  & 0 & 0  \\
				\end{array}\right]_{{n\times n} }}_{A_{21}}
		\end{array}
		\right]_{2n\times2n},
	\end{align}
	be the corresponding companion matrix, partitioned as what it is.
	Constructing the Cartesian decomposition of $C\left(q\right)$, we have
	\begin{align}	
		\label{eq3.14}{\rm{Re}}\left( C\left(q\right)\right)=
		\left[
		\begin{array}{c|c}
			\overbrace{\left[\begin{array}{*{20}c}
					{ - {\mathop{\rm Re}\nolimits} \left( {a_{2n} } \right)} & {\frac{{ - a_{2n - 1}  + 1}}{2}} &  \cdots  & {\frac{{ - a_{n + 2} }}{2}} & {\frac{{ - a_{n + 1} }}{2}} 
					\\
					{\frac{{ - \overline {a_{2n - 1} }  + 1}}{2}} & 0 &  \ddots  & 0 & 0 
					\\
					\vdots  & {\frac{1}{2}} &  \ddots  & {\frac{1}{2}} &  \vdots  
					\\
					0 & 0 &  \ddots  & 0 & \frac{1}{2}  
					\\
					{\frac{{ - \overline {a_{n + 1} } }}{2}} & 0 &  \cdots  & {\frac{1}{2}} & 0 \\
				\end{array}\right]_{n\times n}}^{P_{11}} & 
			\overbrace{\left[\begin{array}{*{20}c}
					{\frac{{ - a_n }}{2}} & {\frac{{ - a_{n - 1} }}{2}} &  \cdots  & {\frac{{ - a_2 }}{2}} & {\frac{{ - a_1 }}{2}}  
					\\
					0 & 0 &  \cdots  & 0 & 0 
					\\
					0 & 0 &  \cdots  & 0 & 0 
					\\
					\vdots & \vdots &  \vdots  & \vdots &\vdots 
					\\
					{\frac{1}{2}} & 0 &  \cdots  & 0 & 0  
					\\
				\end{array}\right]_{n\times n}}^{P_{12}}
			\\\hline
			\underbrace{\left[\begin{array}{*{20}c}
					{\frac{{ - \overline {a_n } }}{2}} & 0 &  \cdots  &  0 & {\frac{1}{2}} 
					\\
					{\frac{{ - \overline {a_{n - 1} } }}{2}} & 0 &  \cdots  &  \cdots  & 0 
					\\
					\vdots  & 0 &  \ddots  &  \cdots  & 0
					\\
					{\frac{{ - \overline {a_2 } }}{2}} &  \vdots  &  \cdots  &  \cdots  &  \vdots  
					\\
					{\frac{{ - \overline {a_1 } }}{2}} & 0 &  \cdots  & 0 & 0  \\
				\end{array}\right]_{{n\times n} }}_{P_{21}}& 
			\underbrace{\left[\begin{array}{*{20}c}
					0 & {\frac{1}{2}} & 0  &  \cdots  & 0  
					\\
					{\frac{1}{2}} & 0 & {\frac{1}{2}} & 0 &  \vdots   
					\\
					0 & {\frac{1}{2}} &  \ddots  &  \ddots  & 0  
					\\
					\vdots  &  \ddots  &  \ddots  &  \ddots  & {\frac{1}{2}}  
					\\
					0  &  \cdots  & 0 & {\frac{1}{2}} & 0  \\
				\end{array}\right]_{{n\times n} }}_{P_{22}}
		\end{array}
		\right]_{2n\times2n}
	\end{align}
	and
	\begin{align}	
		\label{eq3.15}{\rm{Im}}\left( C\left(q\right)\right)=
		\left[
		\begin{array}{c|c}
			\overbrace{\left[\begin{array}{*{20}c}
					{ - {\mathop{\rm Im}\nolimits} \left( {a_{2n} } \right)} & {\frac{{ - a_{2n - 1}  -1}}{2i}} &  \cdots  & {\frac{{ - a_{n + 2} }}{2i}} & {\frac{{ - a_{n + 1} }}{2i}} 
					\\
					{\frac{{  \overline {a_{2n - 1} }  + 1}}{2i}} & 0 &  \ddots  & 0 & 0 
					\\
					\vdots  & {\frac{1}{2i}} &  \ddots  & {\frac{-1}{2i}} &  \vdots  
					\\
					0 & 0 &  \ddots  & 0 & \frac{-1}{2i}  
					\\
					{\frac{{   \overline {a_{n + 1} } }}{2i}} & 0 &  \cdots  & {\frac{1}{2i}} & 0 \\
				\end{array}\right]_{n\times n}}^{Q_{11}} & 
			\overbrace{\left[\begin{array}{*{20}c}
					{\frac{{ - a_n }}{2i}} & {\frac{{   a_{n - 1} }}{2i}} &  \cdots  & {\frac{{   a_2 }}{2i}} & {\frac{{   a_1 }}{2i}}  
					\\
					0 & 0 &  \cdots  & 0 & 0 
					\\
					0 & 0 &  \cdots  & 0 & 0 
					\\
					\vdots & \vdots &  \vdots  & \vdots &\vdots 
					\\
					{\frac{-1}{2i}} & 0 &  \cdots  & 0 & 0  
					\\
				\end{array}\right]_{n\times n}}^{Q_{12}}
			\\\hline
			\underbrace{\left[\begin{array}{*{20}c}
					{\frac{{  \overline {a_n } }}{2i}} & 0 &  \cdots  &  0 & {\frac{1}{2i}} 
					\\
					{\frac{{  \overline {a_{n - 1} } }}{2i}} & 0 &  \cdots  &  \cdots  & 0 
					\\
					\vdots  & 0 &  \ddots  &  \cdots  & 0
					\\
					{\frac{{  \overline {a_2 } }}{2i}} &  \vdots  &  \cdots  &  \cdots  &  \vdots  
					\\
					{\frac{{  \overline {a_1 } }}{2i}} & 0 &  \cdots  & 0 & 0  \\
				\end{array}\right]_{{n\times n} }}_{Q_{21}} & 
			\underbrace{\left[\begin{array}{*{20}c}
					0 & {\frac{-1}{2i}} & 0  &  \cdots  & 0  
					\\
					{\frac{1}{2i}} & 0 & {\frac{-1}{2i}} & 0 &  \vdots   
					\\
					0 & {\frac{1}{2i}} &  \ddots  &  \ddots  & 0  
					\\
					\vdots  &  \ddots  &  \ddots  &  \ddots  & {\frac{-1}{2i}}  
					\\
					0  &  \cdots  & 0 & {\frac{1}{2i}} & 0  \\
				\end{array}\right]_{{n\times n} }}_{Q_{22}}
		\end{array}
		\right]_{2n\times 2n}.
	\end{align}	 
	Hence, 	 
	\begin{align*}
		C\left(q\right) :={\rm{Re}}\left( C\left(q\right)\right)+i{\rm{Im}}\left( C\left(q\right)\right).
	\end{align*}	 
	Moreover, it is easy to observe that $A_{kj}=P_{kj}+iQ_{kj}$ for $k,j=1,2$. This observation may not hold true in general for operator matrices; i.e., the 
	Cartesian decomposition of operator matrices is not equal to the Cartesian decomposition of their sub-matrices. However, since we construct a special type of partition of $C\left(q\right)$, it seems we have such equality holds true only for this construction.\\
	
	It's not easy to apply \eqref{eq2.2}  for general entries, since it has the norms $\left\| \left| {P_{11} } \right| + \left| {Q_{11} } \right| \right\|$ and $\left\| \left| {P_{22} } \right| + \left| {Q_{22} } \right| \right\|$; which are difficult to evaluate for general $n\times n$ matrices. In fact, it will be more easy as long as we have numeric entries with specific $n$ as explored in the presented examples below.  \\

 Table \ref{tab1} illustrates that our upper bound of any zero of $q\left(z\right)  = z^6  + \frac{5}{4}z^5  + \frac{4}{3}z^4  +  z^3  + 2z^2  + 3z + 4$, obtained by \eqref{eq2.2} is much better among all given upper bounds listed in  Table \ref{tab1}.\\
	
	\begin{table}
			\caption{\label{Table 1} \label{tab1} } 
		\begin{tabular}{ c c }
			\hline
			{\bf Mathematician}  \qquad &  \qquad {\bf Upper bound} \\\hline
			Cauchy \eqref{eq3.3}\qquad & \qquad $5$  \\    
			Carmichael and Mason \eqref{eq3.4} \qquad &  \qquad $5.860057831$ \\
			Montel \eqref{eq3.5} \qquad & \qquad $12.58333333$  \\
			Fujii and Kubo \eqref{eq3.6}\qquad & \qquad $18.19610776$  \\  
			Abdurakhmanov \eqref{eq3.7} \qquad &\qquad$17.44802607$\\
			Linden \eqref{eq3.8}\qquad &\qquad$5.845408848$\\
			Kittaneh \eqref{eq3.9}\qquad &\qquad $4.040959271$\\
			Abu-omar and Kittaneh \eqref{eq3.10} \qquad &\qquad $4.916052295$\\
			Al-Dolat \etal  \eqref{eq3.11}\qquad & \qquad $4.867955746$  \\
			{\bf Corollary \ref{cor1}}  \qquad & \qquad ${\bf3.941508802}$  \\ 
				\hline
		\end{tabular}
\end{table}
	\centerline{}\centerline{}

	We remark that, the same approach can be applied for polynomials of odd degrees having their absolute terms zero; e.g., in \eqref{eq3.12} assume that $n$ is odd $\ge5$ and $a_1=0$. Then, $p\left(z\right)$ can be written as 
	\begin{align*}
		p\left(z\right)=z\left(z^{n-1} +a_nz^{n-2}+\cdots+a_2 \right)=zp_1\left(z\right),
	\end{align*}
	where $p_1\left(z\right)$ is an even polynomial of degree $\ge4$. Since $z=0$ is a zero for $p$, then trivially it must belongs to the disk containing  the zeros of $p_1$. Hence, in this case we have $\omega\left(C\left(p\right)\right)=\omega\left(C\left(p_1\right)\right)$. We left the details to the interested reader.


	\section{ Numerical radius of real and Imaginary parts of $C\left(q\right)$}
	
	Let $T\in\mathscr{M}_n\left(\mathbb{C}\right)$ with the Cartesian decomposition $T=P+iQ$. Then
	\begin{align*}
		W\left(T\right)\subseteq W\left(P\right)+W\left(Q\right).
	\end{align*}
	Hence, 
	\begin{align*}
		\sigma\left(T\right)	\subseteq
		\left[ {\lambda _{\min } \left( P \right),\lambda _{\max } \left( P \right)} \right] \times \left[ {\lambda _{\min } \left( Q \right),\lambda _{\max } \left( Q \right)} \right].
	\end{align*}
	In case of the companion matrix $C\left(q\right)$, it follows that all the zeros of $p\left(z\right)$ in \eqref{eq3.1} are located in the rectangle 
	\begin{align}
		\label{eq4.1}\left[ {\lambda _{\min } \left( {\rm{Re}}\left[C\left(q\right)\right] \right),\lambda _{\max } \left( {\rm{Re}}\left[C\left(q\right)\right] \right)} \right] \times \left[ {\lambda _{\min } \left( {\rm{Im}}\left[C\left(q\right)\right] \right),\lambda _{\max } \left( {\rm{Im}}\left[C\left(q\right)\right] \right)} \right].
	\end{align}
	In \cite{FK2}, Kittaneh provided an explicit formula for the characteristic polynomial of 
	$ {\rm{Re}}\left[C\left(q\right)\right]$, which given as:
	\begin{align*}
		p_{{\mathop{\rm Real}\nolimits} } \left( z \right)
		:=
		\left( {z + {\mathop{\rm Re}\nolimits} \left( {a_n } \right)} \right)\prod\limits_{j = 1}^{n - 1} {\left( {z - \cos \left( {\frac{{j\pi }}{n}} \right)} \right)}  - \sum\limits_{j = 1}^{n - 1} {\left( {\prod\limits_{k \ne j}^{n - 1} {\left( {z - \cos \left( {\frac{{k\pi }}{n}} \right)} \right)} } \right)\left| {v_j } \right|^2 }, 
	\end{align*}
	where 
	\begin{align*}
		v_j  = \frac{1}{{\sqrt {2n} }}\left[ {\left( {1 - \overline a _{n - 1} } \right)\sin \left( {\frac{{j\pi }}{n}} \right) - \sum\limits_{k = 2}^{n - 1} {\overline a _{n - k} \sin \left( {\frac{{kj\pi }}{n}} \right)} } \right].
	\end{align*}

	In the same work \cite{FK2}, an explicit rectangle that contains the rectangle \eqref{eq4.1},
	and thus it contains all the zeros of $p$, is obtained in the following result.
	\begin{theorem}\cite[Kittaneh]{FK2}
		\label{thm2}If $z$ is any zero of $p$, then $z$ belongs to the rectangle $\left[-c,c\right]\times \left[-d,d\right]$, where
		\begin{align*}
			c= \frac{1}{2}\left[ {\left| {{\mathop{\rm Re}\nolimits} \left( {c_n } \right)} \right| + \cos \left( {\frac{\pi }{n}} \right) + \sqrt {\left( {\left| {{\mathop{\rm Re}\nolimits} \left( {c_n } \right)} \right| - \cos \left( {\frac{\pi }{n}} \right)} \right)^2  + \left| {c_{n - 1}  - 1} \right|^2  + \sum\limits_{k = 1}^{n - 2} {\left| {c_k } \right|^2 } } } \right]	
		\end{align*}
		and
		\begin{align*}
			d= \frac{1}{2}\left[ {\left| {{\mathop{\rm Im}\nolimits} \left( {c_n } \right)} \right| + \cos \left( {\frac{\pi }{n}} \right) + \sqrt {\left( {\left| {{\mathop{\rm Im}\nolimits} \left( {c_n } \right)} \right| - \cos \left( {\frac{\pi }{n}} \right)} \right)^2  + \left| {c_{n - 1}  - 1} \right|^2  + \sum\limits_{k = 1}^{n - 2} {\left| {c_k } \right|^2 } } } \right].	
		\end{align*}	
	\end{theorem}
	
	Based on the results obtained in this work, in what follows, we provide another possible rectangle. In order to establish our result, we need the following lemma \cite{DJH}.
	\begin{lemma}
		\label{lemma5}Let $A,B,C,D \in \mathscr{B\left(H\right)}$ and $T=\left[ {\begin{array}{*{20}c}
				A & B  \\
				C & D  \\
		\end{array}} \right]$. Then
		\begin{align*}
			\omega \left( T \right) \le \frac{1}{2}\left( {\omega \left( A \right) + \omega \left( D \right) + \sqrt {\left( {\omega \left( A \right) - \omega \left( D \right)} \right)^2  + \left( {\omega \left( {B + C} \right) + \omega \left( {B - C} \right)} \right)^2 } } \right).
		\end{align*}
	\end{lemma}
	
	Now, we are in position to give our explicit rectangle that contains the rectangle \eqref{eq4.1},
	and thus it contains all the zeros of $p$.
	\begin{theorem}
		\label{thm3} Let $q\left(z\right)$ be any even complex polynomial whose degree $\ge4$. If $z$ is any zero of $q$, then $z$ belongs to the rectangle $\left[-s,s\right]\times \left[-t,t\right]$, where  
		\begin{align*}
			s&:= \frac{1}{4} \left(\left| {{\mathop{\rm Re}\nolimits} \left( {a_{2n} } \right) } \right|+ \cos\left(\frac{\pi}{n}\right)+F\right)+\frac{1}{2}\cos\left(\frac{\pi}{n+1}\right)  
			\\
			&\qquad+\frac{1}{2}\sqrt{\left(\frac{1}{2} \left(\left| {{\mathop{\rm Re}\nolimits} \left( {a_{2n} } \right) } \right|+ \cos\left(\frac{\pi}{n}\right)+F\right)-\cos\left(\frac{\pi}{n+1}\right)\right)^2+\left(\frac{ \left| {{\mathop{\rm Re}\nolimits} \left( {a_{n} } \right) } \right| +G+  \left| {{\mathop{\rm Im}\nolimits} \left( {a_{n} } \right) } \right|+H }{2}\right)^2 },\nonumber
		\end{align*} 
		and
		\begin{align*}
			t&:= \frac{1}{4} \left(\left| {{\mathop{\rm Im}\nolimits} \left( {a_{2n} } \right) } \right|+ \cos\left(\frac{\pi}{n}\right)+J\right)+\frac{1}{2}\cos\left(\frac{\pi}{n+1}\right)  
			\\
			&\qquad+\frac{1}{2}\sqrt{\left(\frac{1}{2} \left(\left| {{\mathop{\rm Im}\nolimits} \left( {a_{2n} } \right) } \right|+ \cos\left(\frac{\pi}{n}\right)+J\right)-\cos\left(\frac{\pi}{n+1}\right)\right)^2+\left(\frac{ \left| {{\mathop{\rm Re}\nolimits} \left( {a_{n} } \right) } \right| +G+  \left| {{\mathop{\rm Im}\nolimits} \left( {a_{n} } \right) } \right|+H }{2}\right)^2 },\nonumber
		\end{align*} 
		where 
		\begin{align*}
			F&:= \sqrt{\left(\left| {{\mathop{\rm Re}\nolimits} \left( {a_{2n} } \right) } \right|- \cos\left(\frac{\pi}{n}\right)\right)^2+ 	\left| {1 - a_{2n - 1}   } \right|^2 +  \sum\limits_{k = n + 1}^{2n - 2} {\left| {a_k } \right|^2 }    },
			\\
			G&:= \sqrt{ \left| {{\mathop{\rm Re}\nolimits} \left( {a_{n} } \right) } \right| ^2+  	\left| {1 - a_{ 1}   } \right|^2 +   \sum\limits_{k = 2}^{n-1} {\left| {a_k } \right|^2 }    },
			\\
			H&:= \sqrt{ \left| {{\mathop{\rm Im}\nolimits} \left( {a_{n} } \right) } \right| ^2+  	\left| {1 + a_{ 1}   } \right|^2 +   \sum\limits_{k = 2}^{n-1} {\left| {a_k } \right|^2 }    },
		\end{align*} 
		and
		\begin{align*}
			J&:= \sqrt{\left(\left| {{\mathop{\rm Im}\nolimits} \left( {a_{2n} } \right) } \right|- \cos\left(\frac{\pi}{n}\right)\right)^2+ 	\left| {1 +a_{2n - 1}   } \right|^2 +  \sum\limits_{k = n + 1}^{2n - 2} {\left| {a_k } \right|^2 }    }.	
		\end{align*}

	\end{theorem}
	
	\begin{proof}
		Employing Lemma \ref{lemma5}, on the real part  of $C\left(q\right)$ \eqref{eq3.14}; which is obtained in the previous section,  by setting $A=P_{11}$, $B=P_{12}$, $C=P_{21}$ and $D=P_{22}$, it's enough to show that
		\begin{align}
			\label{eq4.2}	\omega \left( {\rm{Re}} \left[C\left(q\right) \right]\right) \le \frac{1}{2}\left( {\omega \left( P_{11} \right) + \omega \left( P_{22}\right) + \sqrt {\left( {\omega \left( P_{11} \right) - \omega \left( P_{22} \right)} \right)^2  + \left( {\omega \left( {P_{12} + P_{21}} \right) + \omega \left( {P_{12} - P_{21}} \right)} \right)^2 } } \right).
		\end{align}
		Let us simplify that, indeed we have
		
		\begin{align*}
			P_{11} =\left[
			\begin{array}{c|c}
				\begin{array}{*{20}c}
					{ - {\mathop{\rm Re}\nolimits} \left( {a_{2n} } \right)} \\
				\end{array}  & 
				\begin{array}{*{20}c}
					{\frac{{ - a_{2n - 1}  + 1}}{2}} &  \cdots  & {\frac{{ - a_{n + 2} }}{2}} & {\frac{{ - a_{n + 1} }}{2}} 
					\\
				\end{array} 
				\\\hline
				\begin{array}{*{20}c}
					{\frac{{ - \overline {a_{2n - 1} }  + 1}}{2}}
					\\
					\vdots  
					\\
					{\frac{{ - \overline {a_{n + 2} } }}{2}}
					\\
					{\frac{{ - \overline {a_{n + 1} } }}{2}}  \\
				\end{array} & 
				\begin{array}{*{20}c}
					0 & {\frac{1}{2}} & 0  &  \cdots  & 0  
					\\
					{\frac{1}{2}} & 0 & {\frac{1}{2}} & 0 &  \vdots   
					\\
					0 & {\frac{1}{2}} &  \ddots  &  \ddots  & 0  
					\\
					\vdots  &  \ddots  &  \ddots  &  \ddots  & {\frac{1}{2}}  
					\\
					0  &  \cdots  & 0 & {\frac{1}{2}} & 0  \\
				\end{array}
			\end{array}
			\right]
		\end{align*}
		which is can be written as 
		\begin{align*}
			P_{11}= 
			\left[ {\begin{array}{*{20}c}
					{ - {\mathop{\rm Re}\nolimits} \left( {a_{2n} } \right)}  & u^*  \\
					u  & 	T_{n-1}  \\
			\end{array}} \right]
		\end{align*}
		where $x:=\left[{\frac{{ - \overline {a_{2n - 1} }  + 1}}{2}}
		,
		\cdots  
		,
		{\frac{{ - \overline {a_{n + 2} } }}{2}}
		,
		{\frac{{ - \overline {a_{n + 1} } }}{2}}\right] $. Thus, be employing Lemmas \ref{lem2} \& \ref{lem3},  we have
		 \begin{align}
			\label{eq4.3}\omega\left(P_{11}\right)
			&=\omega\left(
			\left[ {\begin{array}{*{20}c}
					{\left| {{\mathop{\rm Re}\nolimits} \left( {a_{2n} } \right) } \right|} & {x^* }  \\
					x & T_{n-1}    \\
			\end{array}} \right]\right)
			\\
			&\le \omega\left(
			\left[ {\begin{array}{*{20}c}
					{\omega\left(\left| {{\mathop{\rm Re}\nolimits} \left( {a_{2n} } \right) } \right|\right)} & {\left\|x\right\| }  \\
					\left\|x\right\| & \omega\left( T_{n-1}\right) \\
			\end{array}} \right]\right) 
			\nonumber\\
			&=r\left(
			\left[ {\begin{array}{*{20}c}
					{\omega\left(\left| {{\mathop{\rm Re}\nolimits} \left( {a_{2n} } \right) } \right|\right)} & {\left\|x\right\| }  \\
					\left\|x\right\| & \omega\left( T_{n-1}\right) \\
			\end{array}} \right]\right) 
			\nonumber	\\
			&= \frac{1}{2} \left(\left| {{\mathop{\rm Re}\nolimits} \left( {a_{2n} } \right) } \right|+\cos\left(\frac{\pi}{n}\right) +\sqrt{ \left(\left| {{\mathop{\rm Re}\nolimits} \left( {a_{2n} } \right) } \right|-\cos\left(\frac{\pi}{n}\right)\right) ^2+ 4
				\left\|x\right\|^2}\right)
			\nonumber\\
			&= \frac{1}{2} \left(\left| {{\mathop{\rm Re}\nolimits} \left( {a_{2n} } \right) } \right|+\cos\left(\frac{\pi}{n}\right) +\sqrt{ \left(\left| {{\mathop{\rm Re}\nolimits} \left( {a_{2n} } \right) } \right|-\cos\left(\frac{\pi}{n}\right)\right) ^2+ 
				\left| {1 - a_{2n - 1} } \right|^2 + \sum\limits_{k = n + 1}^{2n - 2} {\left| {a_k } \right|^2 }	 }\right).\nonumber
		\end{align}
		On the other hand, we have
		\begin{align*}
			P_{12}+P_{21}&=\left[\begin{array}{*{20}c}
				{\frac{{ - a_n }}{2}} & {\frac{{ - a_{n - 1} }}{2}} &  \cdots  & {\frac{{ - a_2 }}{2}} & {\frac{{ - a_1 }}{2}}  
				\\
				0 & 0 &  \cdots  & 0 & 0 
				\\
				0 & 0 &  \cdots  & 0 & 0 
				\\
				\vdots & \vdots &  \vdots  & \vdots &\vdots 
				\\
				{\frac{1}{2}} & 0 &  \cdots  & 0 & 0  
				\\
			\end{array}\right]+ 
			\left[\begin{array}{*{20}c}
				{\frac{{ - \overline {a_n } }}{2}} & 0 &  \cdots  &  0 & {\frac{1}{2}} 
				\\
				{\frac{{ - \overline {a_{n - 1} } }}{2}} & 0 &  \cdots  &  \cdots  & 0 
				\\
				\vdots  & 0 &  \ddots  &  \cdots  & 0
				\\
				{\frac{{ - \overline {a_2 } }}{2}} &  \vdots  &  \cdots  &  \cdots  &  \vdots  
				\\
				{\frac{{ - \overline {a_1 } }}{2}} & 0 &  \cdots  & 0 & 0  \\
			\end{array}\right]
			\\
			&=\left[
			\begin{array}{c|c}
				\begin{array}{*{20}c}
					{-\frac{{ a_n+\overline {a_n } }}{2}} \\
				\end{array}  & 
				\begin{array}{*{20}c}
					{-\frac{{   a_{n - 1} }}{2 }} &  \cdots  & {-\frac{{   a_2 }}{2 }} & {-\frac{{   a_1 }}{2 }+\frac{1}{2}}  
					\\
				\end{array} 
				\\\hline
				\begin{array}{*{20}c}
					{-\frac{{  \overline {a_{n - 1} } }}{2 }}  
					\\
					\vdots  
					\\
					{-\frac{{  \overline {a_2 } }}{2 }}  
					\\
					{-\frac{{  \overline {a_1 } }}{2 }+\frac{1}{2}}  \\
				\end{array} & 
				\begin{array}{*{20}c}
					0 &    0&\cdots  & 0 
					\\
					\vdots  & \ddots &      \cdots  & 0
					\\
					0 &  \cdots  &  \cdots  &  \vdots   
					\\
					0 & 0 &  \cdots  & 0   \\
				\end{array} 
			\end{array}
			\right]_{n\times n}.
		\end{align*}
		Let $v:=\left[
		{\frac{{ - \overline {a_{ n - 1} }   }}{2}},     \cdots  , {\frac{{ - \overline {a_{  2} } }}{2}} , {\frac{{ - \overline {a_{ 1}  } +1}}{2}} 
		\right]^T$.
		So that, by Lemmas \ref{lem2} \& \ref{lem3},  we have 
		\begin{align}
			\label{eq4.4}\omega\left(P_{12}+P_{21}\right)
			&=\omega\left(
			\left[ {\begin{array}{*{20}c}
					{\left| {{\mathop{\rm Re}\nolimits} \left( {a_{n} } \right) } \right|} & {v^* }  \\
					v & 0  \\
			\end{array}} \right]\right)
			\\
			&\le \omega\left(
			\left[ {\begin{array}{*{20}c}
					{w\left(\left| {{\mathop{\rm Re}\nolimits} \left( {a_{n} } \right) } \right|\right)} & {\left\|v\right\| }  \\
					\left\|v\right\| & 0  \\
			\end{array}} \right]\right) 
			\nonumber\\
			&=r\left(
			\left[ {\begin{array}{*{20}c}
					{w\left(\left| {{\mathop{\rm Re}\nolimits} \left( {a_{n} } \right) } \right|\right)} & {\left\|y\right\| }  \\
					\left\|y\right\| & 0  \\
			\end{array}} \right]\right) 
			\nonumber	\\
			&= \frac{1}{2} \left(\left| {{\mathop{\rm Re}\nolimits} \left( {a_{n} } \right) } \right| +\sqrt{ \left| {{\mathop{\rm Re}\nolimits} \left( {a_{n} } \right) } \right| ^2+4
				\left\|v\right\|^2}\right)
			\nonumber\\
			&= \frac{1}{2} \left(\left| {{\mathop{\rm Re}\nolimits} \left( {a_{n} } \right) } \right| +\sqrt{ \left| {{\mathop{\rm Re}\nolimits} \left( {a_{n} } \right) } \right| ^2+  	\left| {1 - a_{ 1}   } \right|^2 +   \sum\limits_{k = 2}^{n-1} {\left| {a_k } \right|^2 }    }\right).\nonumber
		\end{align}
		
		Similarly, we have
		\begin{align*}
			P_{12}-P_{21}=\left[
			\begin{array}{c|c}
				\begin{array}{*{20}c}
					{-\frac{{ a_n-\overline {a_n } }}{2}} \\
				\end{array}  & 
				\begin{array}{*{20}c}
					{-\frac{{   a_{n - 1} }}{2 }} &  \cdots  & {-\frac{{   a_2 }}{2 }} & {-\frac{{   a_1 }}{2 }-\frac{1}{2}}  
					\\
				\end{array} 
				\\\hline
				\begin{array}{*{20}c}
					{\frac{{  \overline {a_{n - 1} } }}{2 }}  
					\\
					\vdots  
					\\
					{\frac{{  \overline {a_2 } }}{2 }}  
					\\
					{\frac{{  \overline {a_1 } }}{2 }+\frac{1}{2}}  \\
				\end{array} & 
				\begin{array}{*{20}c}
					0 &    0&\cdots  & 0 
					\\
					\vdots  & \ddots &      \cdots  & 0
					\\
					0 &  \cdots  &  \cdots  &  \vdots   
					\\
					0 & 0 &  \cdots  & 0   \\
				\end{array} 
			\end{array}
			\right]_{n\times n}.
		\end{align*}
		Let $u:=\left[
		{\frac{{  \overline {a_{ n - 1} }   }}{2}},     \cdots  , {\frac{{   \overline {a_{  2} } }}{2}} , {\frac{{   \overline {a_{ 1}  } +1}}{2}} 
		\right]^T$.
		So that, by Lemmas \ref{lem2} \& \ref{lem3}, we have 
		\begin{align}
			\label{eq4.5}	\omega\left(P_{12}-P_{21}\right)
			&=\omega\left(
			\left[ {\begin{array}{*{20}c}
					{\left| {-i{\mathop{\rm Im}\nolimits} \left( {a_{n} } \right) } \right|} & {u^* }  \\
					u & 0  \\
			\end{array}} \right]\right)
			\\
			&\le \omega\left(
			\left[ {\begin{array}{*{20}c}
					{w\left(\left| {{\mathop{\rm Im}\nolimits} \left( {a_{n} } \right) } \right|\right)} & {\left\|u\right\| }  \\
					\left\|u\right\| & 0  \\
			\end{array}} \right]\right) 
			\nonumber\\
			&=r\left(
			\left[ {\begin{array}{*{20}c}
					{w\left(\left| {{\mathop{\rm Im}\nolimits} \left( {a_{n} } \right) } \right|\right)} & {\left\|u\right\| }  \\
					\left\|u\right\| & 0 \\
			\end{array}} \right]\right) 
			\nonumber	\\
			&= \frac{1}{2} \left(\left| {{\mathop{\rm Im}\nolimits} \left( {a_{n} } \right) } \right| +\sqrt{ \left| {{\mathop{\rm Im}\nolimits} \left( {a_{n} } \right) } \right| ^2+4
				\left\|u\right\|^2}\right)
			\nonumber\\
			&= \frac{1}{2} \left(\left| {{\mathop{\rm Im}\nolimits} \left( {a_{n} } \right) } \right| +\sqrt{ \left| {{\mathop{\rm Im}\nolimits} \left( {a_{n} } \right) } \right| ^2+  	\left| {1 + a_{ 1}   } \right|^2 +   \sum\limits_{k = 2}^{n-1} {\left| {a_k } \right|^2 }    }\right).\nonumber
		\end{align}
		Also, by  \eqref{eq1.1} and Lemma \ref{lem1}, we have
		\begin{align}
			\label{eq4.6}\omega\left(P_{22}\right)=\cos\left(\frac{\pi}{n+1}\right).
		\end{align}
		Combining all above inequalities and equalities \eqref{eq4.3}--\eqref{eq4.6} in \eqref{eq4.2} we get the required result in Theorem \ref{thm3}. Following the same steps for $\omega \left( {\rm{Re}} \left[C\left(q\right) \right]\right)$, therefore the proof of Theorem \ref{thm3} is established.
	\end{proof}
	
	The following example illustrated in  Table \ref{tab2}, shows that our estimated rectangle given in Theorem \ref{thm3} might be better than that one given in Theorem \ref{thm2}.
	\begin{example}
		Consider $q\left(z\right)=z^6+2iz^5+4i z^4+\frac{1}{4}z+ \frac{1}{16}$, then the real and  imaginary parts of the zeros of $q$,  are bounded as obtained in  Table \ref{tab2}.
		
		\begin{table}
			\caption{\label{Table 2} \label{tab2}} 
			\begin{tabular}{ c c c}
				\hline
				{\bf Result}  \qquad &  \qquad {\bf Upper bound of $	\left| {{\mathop{\rm Re}\nolimits} \left( \lambda  \right)} \right|$}  \qquad &  \qquad {\bf Upper bound of $	\left| {{\mathop{\rm Im}\nolimits} \left( \lambda  \right)} \right|$}
				\\\hline
				Kittaneh \cite{FK2} \qquad &\qquad  $3.999737494$ \qquad &  \qquad$3.576384821$ \\  
				Theorem \ref{thm3}    \qquad & \qquad  $2.476786336$ \qquad &  \qquad $2.585204772$  \\  
				\hline
			\end{tabular}
		\end{table}
		
	\end{example}
	
	\begin{remark}
		\label{rem1}Several particular cases of Theorem \ref{thm3} which are of great interest could be deduced. Among others, we note the following cases:
		\begin{itemize}
			\item $a_{2n}=a_{n}=0$, $a_{2n-1}=\pm 1$ and $a_1=1$.
			
			\item $a_{2n}=a_{n}=0$, $a_{2n-1}=\pm 1$ and $a_k=0$ for all $k=2,3,\cdots, n-1$. In particular, take $a_1=1$.
			
			\item $a_{2n}=a_{n}=0$, $a_{2n-1}=\pm 1$ and $a_k=0$ for all $k=n+1,n+2,\cdots, 2n-2$. In particular, take $a_1=1$.
		\end{itemize}
		
	\end{remark}

	\begin{theorem}
		\label{thm4}Under the assumption of Theorem \ref{thm3}. Then
		\begin{align}
			\label{eq4.7}	\omega \left( C\left(q\right) \right) \le \frac{1}{2}\left( {L+\cos \left( \frac{\pi}{n+1} \right)+ \sqrt {\left( {L -\cos \left( \frac{\pi}{n+1} \right)} \right)^2  + \left( {D_1+D_2} \right)^2 } } \right).
		\end{align}
		where
		\begin{align*}
			L:= \frac{1}{2} \left(\sqrt {\sum\limits_{k = n + 2}^{2n} {\left| {a_k } \right|^2 } }  +\sqrt{ \sum\limits_{k = n + 2}^{2n} {\left| {a_k } \right|^2 }  + 
				\left(\left|a_{n+1}\right| +1\right)^2 }\right),
		\end{align*}	
		\begin{align*}
			D_1	=   \frac{1}{2} \left(\left| {a_{n} } \right| +\sqrt{ \left| {a_{n} } \right| ^2+ \left| {1-a_{1} } \right|^2 +
				\sum\limits_{k = 2}^{n-1} {\left| {a_k } \right|^2 }}\right),
		\end{align*}
		and
		\begin{align*}
			D_2	=   \frac{1}{2} \left(\left| {a_{n} } \right| +\sqrt{ \left| {a_{n} } \right| ^2+ \left| {1+a_{1} } \right|^2 +
				\sum\limits_{k = 2}^{n-1} {\left| {a_k } \right|^2 }}\right),
		\end{align*}
	\end{theorem}	
	
	\begin{proof}
		Applying \eqref{eq4.2} to $C\left(q\right)$ given in \eqref{eq3.13} by setting $A=A_{11}$, $B=A_{12}$, $C=A_{21}$ and $D=A_{22}$. So that, we get
		
		\begin{align}
			\label{eq4.8}\omega \left( C\left(q\right) \right) \le \frac{1}{2}\left( {\omega \left( A_{11} \right) + \omega \left( A_{22} \right) + \sqrt {\left( {\omega \left( A_{11} \right) - \omega \left( A_{22} \right)} \right)^2  + \left( {\omega \left( {A_{12} + A_{21}} \right) + \omega \left( {A_{12} - A_{21}} \right)} \right)^2 } } \right).
		\end{align}
		Let us observe that
		\begin{align*}
			A_{11}=\left[
			\begin{array}{c|c}
				\begin{array}{*{20}c}
					{-a_{2n}} &  	{-a_{2n-1}} &  \cdots  & {-a_{n+2}}\\
				\end{array} &
				\begin{array}{*{20}c}
					{-a_{n+1}}  
					\\
				\end{array} 
				\\\hline
				\begin{array}{*{20}c}
					1 &    0&\cdots  & 0 
					\\
					\vdots  & \ddots &      \cdots  & 0
					\\
					0 &  \cdots  &  \ddots  &  \vdots   
					\\
					0 & 0 &  \cdots  & 1   \\
				\end{array}  & 
				\begin{array}{*{20}c}
					0  
					\\
					\vdots  
					\\
					0    
					\\
					0 \\
				\end{array} 
			\end{array}
			\right]_{n\times n}.
		\end{align*}

		Let $d:=\left[	{-a_{2n}}    \cdots    {-a_{n+2}}  
		\right]$.
		So that, by Lemmas \ref{lem2} \& \ref{lem3}, we have 
		\begin{align}
			\label{eq4.9}	\omega\left(A_{11} \right)
			&=\omega\left(
			\left[ {\begin{array}{*{20}c}
					{ d  } & {-a_{n+1} }  \\
					I & 0  \\
			\end{array}} \right]\right)
			\nonumber\\
			&\le \omega\left(
			\left[ {\begin{array}{*{20}c}
					{  \left\| {  d  } \right\| } & {\left|a_{n+1}\right| }  \\
					\left\|I\right\| & \|0\|  \\
			\end{array}} \right]\right) 
			\nonumber\\
			&=r\left(
			\left[ {\begin{array}{*{20}c}
					{\|d\|} & \frac{\left|a_{n+1}\right| +1}{2}   \\
					\frac{\left|a_{n+1}\right| +1}{2} & 0 \\
			\end{array}} \right]\right) 
			\nonumber	\\
			&= \frac{1}{2} \left(\sqrt {\sum\limits_{k = n + 2}^{2n} {\left| {a_k } \right|^2 } }  +\sqrt{ \sum\limits_{k = n + 2}^{2n} {\left| {a_k } \right|^2 }  + 
				\left(\left|a_{n+1}\right| +1\right)^2 }\right).
		\end{align}
		Now, let $b:=\left[	{-a_{n-1} ,  \cdots  , {-a_{2}}, {1-a_{1}} } \right] $.
		So that, by Lemma \ref{lem3}, we have 
		\begin{align*}
			A_{12}+	A_{21}=\left[
			\begin{array}{c|c}
				\begin{array}{*{20}c}
					{-a_{n}} \\
				\end{array}  & 
				\begin{array}{*{20}c}
					{-a_{n-1}} &  \cdots  & {-a_{2}} & {1-a_{1}}  
					\\
				\end{array} 
				\\\hline
				\begin{array}{*{20}c}
					0  
					\\
					\vdots  
					\\
					0    
					\\
					0 \\
				\end{array} &
				\begin{array}{*{20}c}
					0 &    0&\cdots  & 0 
					\\
					\vdots  & \ddots &      \cdots  & 0
					\\
					0 &  \cdots  &  \ddots  &  \vdots   
					\\
					0 & 0 &  \cdots  & 0   \\
				\end{array}  
			\end{array}
			\right]_{n\times n},
		\end{align*}
		and so
		\begin{align}
			\omega\left(A_{12}+A_{21} \right)
			=\omega\left(
			\left[ {\begin{array}{*{20}c}
					{\left| {- a_{n}   } \right|} & {b  }  \\
					0 & 0  \\
			\end{array}} \right]\right)
			&\le \omega\left(
			\left[ {\begin{array}{*{20}c}
					{w\left(\left| {   a_{n} } \right|\right)} & {\left\|b\right\| }  \\
					0 & 0  \\
			\end{array}} \right]\right) 
			\nonumber\\
			&=r\left(
			\left[ {\begin{array}{*{20}c}
					{w\left(\left| {   a_{n} } \right|\right)} & \frac{1}{2}\left\|b\right\|  \\
					\frac{1}{2}\left\|b\right\| & 0 \\
			\end{array}} \right]\right) 
			\nonumber	\\
			&= \frac{1}{2} \left(\left| {a_{n} } \right| +\sqrt{ \left| {a_{n} } \right| ^2+ 
				\left\|b\right\|^2 }\right)
			\nonumber\\
			&=  \frac{1}{2} \left(\left| {a_{n} } \right| +\sqrt{ \left| {a_{n} } \right| ^2+ \left| {1-a_{1} } \right|^2 +
				\sum\limits_{k = 2}^{n-1} {\left| {a_k } \right|^2 }}\right). \label{eq4.10}
		\end{align}
		
		Similarly, let $z:=\left[	{a_{n-1} ,  \cdots  , {a_{2}}, {1+a_{1}} } \right] $.
		So that, by Lemma \ref{lem3}, we have 
		\begin{align*}
			A_{12}-	A_{21}=\left[
			\begin{array}{c|c}
				\begin{array}{*{20}c}
					{ a_{n}} \\
				\end{array}  & 
				\begin{array}{*{20}c}
					{ a_{n-1}} &  \cdots  & { a_{2}} & {1+a_{1}}  
					\\
				\end{array} 
				\\\hline
				\begin{array}{*{20}c}
					0  
					\\
					\vdots  
					\\
					0    
					\\
					0 \\
				\end{array} &  \begin{array}{*{20}c}
					0 &    0&\cdots  & 0 
					\\
					\vdots  & \ddots &      \cdots  & 0
					\\
					0 &  \cdots  &  \ddots  &  \vdots   
					\\
					0 & 0 &  \cdots  & 0   \\
				\end{array} 
			\end{array}
			\right]_{n\times n}.
		\end{align*}
		So that
		\begin{align}
			\omega\left(A_{12}-A_{21} \right)
			=\omega\left(
			\left[ {\begin{array}{*{20}c}
					{\left| {  a_{n}   } \right|} & {b  }  \\
					0 & 0  \\
			\end{array}} \right]\right)
			&\le \omega\left(
			\left[ {\begin{array}{*{20}c}
					{w\left(\left| {   a_{n} } \right|\right)} & {\left\|z\right\| }  \\
					0 & 0  \\
			\end{array}} \right]\right) 
			\nonumber\\
			&=r\left(
			\left[ {\begin{array}{*{20}c}
					{w\left(\left| {   a_{n} } \right|\right)} & \frac{1}{2}\left\|z\right\|  \\
					\frac{1}{2} \left\|z\right\| & 0 \\
			\end{array}} \right]\right) 
			\nonumber	\\
			&= \frac{1}{2} \left(\left| {a_{n} } \right| +\sqrt{ \left| {a_{n} } \right| ^2+ 
				\left\|z\right\|^2 }\right)
			\nonumber\\
			&=  \frac{1}{2} \left(\left| {a_{n} } \right| +\sqrt{ \left| {a_{n} } \right| ^2+ \left| {1+a_{1} } \right|^2 +
				\sum\limits_{k = 2}^{n-1} {\left| {a_k } \right|^2 }}\right) \label{eq4.11}
		\end{align}
		Also, since $	\omega\left(A_{12}-A_{21} \right)=  \cos \left( {\frac{\pi }{{n + 1}}} \right)$, then  by substituting \eqref{eq4.9}--\eqref{eq4.11} in \eqref{eq4.8} we get the required result in \eqref{eq4.7}. 
	\end{proof}

	The following example illustrates that our upper bound given in Theorem \ref{thm4} is better than some famous and recent upper bounds obtained in the literature. Any zero of $q\left(z\right)=z^6+\frac{1}{2}z^5+\frac{1}{16}z^2+1$, is bounded by any values given in  Table \ref{tab3} and shows that our presented results could be much better than all compared upper bounds listed in Table \ref{tab2}.\\
	
	\begin{table}
			\caption{\label{Table 3}\label{tab3} } 
		\begin{tabular}{ c c }
			\hline
			{\bf Mathematician}  \qquad &  \qquad {\bf Upper bound} \\\hline
			Cauchy \eqref{eq3.3}\qquad & \qquad $2$  \\    
			Carmichael and Mason \eqref{eq3.4} \qquad &  \qquad $1.501301519$ \\
			Montel \eqref{eq3.5}\qquad & \qquad $1.5625$  \\
			Fujii and Kubo \eqref{eq3.6}\qquad & \qquad $1.777921993$  \\  
			Abdurakhmanov \eqref{eq3.7} \qquad &\qquad$1.701542875$\\
			Linden \eqref{eq3.8}\qquad &\qquad$2.350962955$\\
			Kittaneh \eqref{eq3.9}\qquad &\qquad $1.455651176$\\
			Abu-omar and Kittaneh \eqref{eq3.10} \qquad &\qquad $1.857439836$\\
			Al-Dolat \etal \eqref{eq3.11} \qquad & \qquad $2.147748325$  \\
			\bf{Theorem \ref{thm4}}    \qquad & \qquad $\bf{1.307548659}$  \\
			\hline
			\end{tabular}
	\end{table}

	\begin{corollary}
		\label{cor3}	Under the assumptions of Theorem \ref{thm4}. If $a_k=0$ for all $k=2,3,\cdots, n$ and $a_1=1$ (or $a_1=-1$), then we have
		
		\begin{align}
			\label{eq4.12}	\omega \left( C\left(q\right) \right) \le \frac{1}{2}\left( {L + \cos\left(\frac{\pi}{n+1}\right) + \sqrt {\left( {L -\cos\left(\frac{\pi}{n+1}\right)} \right)^2  +1 } } \right).
		\end{align}
		where $L$ is defined in Theorem \ref{thm4}.
	\end{corollary}
	
	\begin{proof}
		From \eqref{eq4.2} we have $\omega \left( {A_{12} + A_{21}} \right)=0$ and $\omega \left( {A_{12} - A_{21}} \right)=1$, if $a_1=1$. Thus, we have 
		\begin{align*}
			\left( {\omega \left( {A_{12} + A_{21}} \right) + \omega \left( {A_{12} - A_{21}} \right)} \right)^2=1
		\end{align*}
		Also, we always have $\omega \left( A_{22} \right)=\cos\left(\frac{\pi}{n+1}\right)$. Employing \eqref{eq4.2} we get the desired result in \eqref{eq4.12}.
	\end{proof}

	\begin{remark}
		It is convenient to note that \eqref{eq4.7} can be rewritten as	
		\begin{align*}
			\omega \left( C\left(q\right) \right) \le \frac{1}{2}\left( {\omega\left(A_{11}\right)+\cos \left( \frac{\pi}{n+1} \right)+ \sqrt {\left( {\omega\left(A_{11}\right) -\cos \left( \frac{\pi}{n+1} \right)} \right)^2  + \left( {D_1+D_2} \right)^2 } } \right).
		\end{align*}	
		Moreover, 	if $a_{n+1}\ne0$ in Theorem \ref{thm4}, one can replace the upper bound of $\omega\left(A_{11}\right)$ by any other upper bound established in literature.   Indeed, it could be chosen as minimum as possible, and this improves  our result in \eqref{eq4.7}.  
	\end{remark}

	We end this work by giving a new upper bound for the numerical radius of the   companion matrix that represents the polynomial $g\left(z\right)=z^n+c_{n}z^{n-1}+c_{n-1}z^{n-2}+\cdots+ c_2z+c_1$ (with $c_1\ne0$), of any degree $\ge2$. 
	
	Our upper bound is exactly the number $L$ defined in Theorem \ref{thm4}. Naming, $c_n=a_{2n}, c_{n-1}=a_{2n-1}, \cdots, c_{2}=a_{n+2}, c_1=a_{n+1}$. If $c _{k}$'s are all reals such that $|c _{k}|\mathop  < \limits_ {\ne} 1$  and $|c_{k+1}|
	\mathop  > \limits_ {\ne}  |c _{k}|$ $(\forall k=1,2,\cdots,n-1)$. Then, we have
	\begin{align}
		\label{eq4.13}\omega\left(	C\left( g \right)\right)\le \frac{1}{2} \left(\sqrt {\sum\limits_{k =  2}^{n} {\left| {c_k } \right|^2 } }  +\sqrt{ \sum\limits_{k =  2}^{n} {\left| {c_k } \right|^2 }  + 	\left(\left|c_1\right| +1\right)^2 }\right):={\bf MW},
	\end{align}
	provided that $
	\sum\limits_{k = 2}^{n} {\left| {c_k } \right|}\ge  \frac{2}{3}$.  {\bf Otherwise}, the result still valid even we don't have these assumption(s); i.e., if $ |c _{k}| \ge 1$ for some $k\ne 1$, then \eqref{eq4.13} it remains always true for both real and complex  coefficients.
	The analysis of the proof is mentioned in the proof of Theorem \ref{thm4}, as stated for $\omega\left(A_{11}\right)$, under the assumption that $c_1=a_{n+1}\ne0$.  
	
	\section{Observations, Discussion and Conclusion regarding {\bf MW}} 
	
	Consider $f(z)= z^6 + \frac{1}{4}z^5 + \frac{1}{9}z^4  + \frac{1}{16}z^3  + \frac{1}{25}z^2  + \frac{1}{36}z + \frac{1}{49}$, we find that the largest zero has modulus $={\bf 0.5447544053}$. Table \ref{tab4} shows that our result {\bf{MW}} is pretty close to the exact modulus and it is much better than all other upper bounds.\\ 
	
	\begin{table}
			\caption{\label{Table 4}\label{tab4}} 
		\begin{tabular}{ c c }
			\hline
			{\bf Mathematician}  \qquad &  \qquad {\bf Upper bound} \\\hline
			Cauchy \eqref{eq3.3}\qquad & \qquad $1.25$  \\    
			Carmichael and Mason \eqref{eq3.4} \qquad &  \qquad $1.039971167$ \\
			Montel \eqref{eq3.5}\qquad & \qquad $1$  \\
			Fujii and Kubo \eqref{eq3.6}\qquad & \qquad $1.066738881$  \\  
			Abdurakhmanov \eqref{eq3.7} \qquad &\qquad$1.198213950$\\
			Linden \eqref{eq3.8}\qquad &\qquad$2.091031073$\\
			Kittaneh \eqref{eq3.9}\qquad &\qquad $1.152835774$\\
			Abu-omar and Kittaneh \eqref{eq3.10}\qquad &\qquad $1.072449189$\\
			Al-Dolat \etal \eqref{eq3.11} \qquad & \qquad $1.573586825$  \\
			Theorem \ref{thm4}    \qquad & \qquad $1.219108946$  \\
			\bf{MW}    \qquad & \qquad $\bf{0.6721175730}$  \\
			\hline
		\end{tabular}
	\end{table}
	\centerline{}\centerline{}
	
	Another example shows the effienicy of our result \eqref{eq4.13}, consider $
	g\left( z \right) = z^6  + \frac{1}{3}z^4  + \frac{1}{4}z^3  + \frac{1}{9}z^2  + \frac{1}{{100}}$, we find that the largest zero has modulus $={\bf 0.7419983061}$. Table \ref{tab5}  shows that our result {\bf{MW}} is very close to the exact modulus.\\ 
	
	\begin{table}
			\caption{\label{Table 5}\label{tab5}} 
		\begin{tabular}{ c c }
			\hline
			{\bf Mathematician}  \qquad &  \qquad {\bf Upper bound} \\\hline
			Cauchy \eqref{eq3.3}\qquad & \qquad $1.333333333$  \\    
			Carmichael and Mason \eqref{eq3.4} \qquad &  \qquad $1.089062344$ \\
			Montel \eqref{eq3.5}\qquad & \qquad $1$  \\
			Fujii and Kubo \eqref{eq3.6}\qquad & \qquad $0.9939972629$  \\  
			Abdurakhmanov \eqref{eq3.7} \qquad &\qquad$1.167303296$\\
			Linden \eqref{eq3.8}\qquad &\qquad$2.078873251$\\
			Kittaneh \eqref{eq3.9}\qquad &\qquad $1.325435041$\\
			Abu-omar and Kittaneh \eqref{eq3.10}\qquad &\qquad $1.299097566$\\
			Al-Dolat \etal \eqref{eq3.11} \qquad & \qquad $1.581696908$  \\
			Theorem \ref{thm4}    \qquad & \qquad $1.351458429$  \\
			\bf{MW}    \qquad & \qquad $\bf{0.7647166222}$  \\
			\hline
		\end{tabular}
	\end{table}

	After long considerations and investigations, for almost all cases of $c_k$'s (with $|c_k|<1$), we find that \eqref{eq4.13} is very effective as long as $|c_k|<1$, besides the other mentioned assumptions.  The following example explains why we chose the presented conditions in \eqref{eq4.13}, and thus it solidifies and supports the reason of our selection. Let
	\begin{align*}
		h_1(z)= z^6   + \frac{1}{6}z^4  +  \frac{1}{5}z^2  +   \frac{1}{4}. 
	\end{align*}
	The real coefficients of $h_1$ don't respect our conditions in \eqref{eq4.13}. So that, we find that the largest zero has modulus $=0.8120242973$, but the upper bound in \eqref{eq4.13} $=0.7685824855$; which gives an incorrect upper bound, which means that \eqref{eq4.13} doesn't apply arbitrary for any $c_k$'s unless we have some restriction(s). To ensure  the correctness; and after long extrapolation by testing many cases; we established the investigated assumptions in \eqref{eq4.13} for all polynomials with real coefficients.\\
	
	In this regard, we should note that the assumptions about $c_k$'s whenever $|c_k|<1$; are sufficient only for real coefficients. However, if some of $c_k$'s are complex, then it is not true that we cannot apply \eqref{eq4.13}. For example, the polynomial 
	\begin{align*}
		h_2(z)= z^6 + \left(\frac{1}{4}+i\frac{1}{4}\right)z^5 + \frac{i}{9}z^4  + \frac{i}{16}z^3  + \frac{1}{25}z^2  + \frac{1}{36}z + \frac{1}{49}, 
	\end{align*}
	refuting three assumptions of \eqref{eq4.13}. Namely, we have some of $c_k$'s are complex, with
	$\sum\limits_{k = 2}^{n} {\left| {c_k } \right|}=0.5949422794 \le  \frac{2}{3}$, and $|c_3|<|c_2|$. At the same time, we find that the largest zero has modulus $=0.6408240287$. But  the upper bound in \eqref{eq4.13} $=0.7337440145$, which means that \eqref{eq4.13} still can be applied under another conditions as long as $|c_k|<1$. We are not able to determine the sufficient condition for complex coefficients in the case that $|c_k|<1$.
	
	Furthermore, the polynomial 
	\begin{align*}
		h_3(z)= z^6   + \frac{1}{4}z^4  +  \frac{1}{3}z^2  +   \frac{1}{4}, 
	\end{align*}
	is a good example showing that the conditions $|c_{k+1}|<|c_k|$ and $\sum\limits_{k = 2}^{n} {\left| {c_k } \right|}  =0.5833333333  \le\frac{2}{3}$, are not necessary but they are sufficient. In other words, the restricted conditions in \eqref{eq4.13} do not mean that there are no other examples that violate these conditions. Indeed, the largest zero of $h_3$  has modulus $=0.8310538215$, however the upper bound in \eqref{eq4.13} $=0.8671411790$. Therefore, \eqref{eq4.13} can be applied even we don't have our restrictions.
	
	In this matter, we would say our restricted assumptions in \eqref{eq4.13} are very sufficient and hold correctly over all reals as long as we have the mentioned assumptions in \eqref{eq4.13}; but they are not necessary, in general. We left the study  of sufficiency and necessity conditions in the  case that $|c_k|<1$ as an open problem to the interested reader. When some of $|c_k|>1$, $ k\ne 1$, \eqref{eq4.13} is always valid, regardless; what types of the coefficients we have, or their arrangement, or their sum.  Last but not least, the upper bound \eqref{eq4.13} shows a very impact efficiency especially in the case that all of $|c_k|$'s are $<1$.

	\centerline{}\centerline{}

	\noindent \textbf{Conflict of Interest:} The author declares that there is no conflict of interest. 
	

	

	\centerline{}\centerline{}

	\bibliographystyle{amsplain}

\end{document}